\newcommand{\set}[1]{\left\{#1\right\}}
\newcommand{\dual}[1]{\langle#1\rangle}

\newcommand{\abs}[1]{\lvert#1\rvert}
\newcommand{\norm}[1]{\lVert#1\rVert}
\newcommand{\trinorm}[1]{{\left\vert\kern-0.25ex\left\vert\kern-0.25ex\left\vert #1 
    \right\vert\kern-0.25ex\right\vert\kern-0.25ex\right\vert}}

\newcommand{\jump}[1]{\llbracket #1 \rrbracket}
\newcommand{\mean}[1]{\{#1\}}

\newcommand{\disp}{\displaystyle}

\newcommand{\nder}[1]{\disp\frac{\partial#1}{\partial \n}}

\newcommand{\n}{\boldsymbol{n}}

\newcommand{\x}{\boldsymbol{x}}
\newcommand{\y}{\boldsymbol{y}}
\newcommand{\bu}{\boldsymbol{u}}
\newcommand{\bv}{\boldsymbol{v}}
\newcommand{\bj}{\boldsymbol{j}}
\newcommand{\bH}{\boldsymbol{h}}
\newcommand{\bE}{\boldsymbol{e}}
\newcommand{\curl}{\mathbf{curl}}
\renewcommand{\div}{\mathbf{div}}

\newcommand{\bPi}{\boldsymbol{\Pi}}

\newcommand{\cT}{\mathcal{T}}
\newcommand{\cF}{\mathcal{F}}

\newcommand{\cP}{\mathcal{P}}

\newcommand{\rmH}{\mathrm{H}}
\newcommand{\rmL}{\mathrm{L}}

\documentclass[12pt]{article}

\usepackage{amsmath,amsthm,amssymb,amsfonts,stmaryrd}
\usepackage{latexsym,graphicx,color}
\newtheorem{remark}{Remark}[section]
\newtheorem{lemma}{Lemma}[section]
\newtheorem{theorem}{Theorem}[section]
\newtheorem{prop}{Proposition}[section]

\date{}
\title{Coupling DG-FEM  and BEM for time harmonic eddy current problem
\thanks{Partial support by the University of Trento, and
Spain's Ministry of Economy through Project MTM2013-43671-P.
}}

\author{{\sc Ana Alonso Rodr\'iguez}\thanks{
Department of Mathematics, University of Trento,
Trento, Italy,
e-mail: {\tt alonso@science.unitn.it}}, $\,\,$
{\sc Salim Meddahi}\thanks{Departamento de Matem\'aticas, Facultad de Ciencias,
Universidad de Oviedo, Calvo Sotelo s/n, Oviedo, Espa\~na,
e-mail: {\tt salim@uniovi.es}}\\ and \\
{\sc Alberto Valli}\thanks{Department of Mathematics, University of Trento,
	Trento, Italy, e-mail: {\tt
valli@science.unitn.it}}
}

\begin{document}

\maketitle

\begin{abstract}
We introduce and analyze a discontinuous Galerkin FEM/BEM method for a time-harmonic eddy current  problem written in terms of the magnetic field. We use nonconforming N\'ed\'elec finite elements on a partition of the interior domain coupled with continuous boundary elements on the transmission interface. We prove quasi-optimal error estimates in the energy norm.
\end{abstract}

\section{Introduction}
\label{sec:1}
The idea of coupling finite elements (FEM) and boundary elements (BEM) to solve eddy current problems has been introduced in \cite{Bossavit91} by Bossavit and more recently in \cite{Hiptmair, MS03}. Our aim here is to revisit the FEM/BEM formulation  given in \cite{MS03} in order to provide an interior penalty discontinuous Galerkin (IPDG) approximation of the magnetic field in the interior domain.

Discontinuos Galerkin (DG) methods can provide efficient solvers for electromagnetic problems posed in complex geometries and requiring $hp$ adaptivity, see \cite{Carstensen}. However,  we only found few works applying DG methods to eddy current problems (see \cite{PS03} for the time-harmonic regime and  \cite{ABP09} for a time-domain problem) and we are not aware about any DG-FEM/BEM formulation for this problem. 
 
Due to the nonlocal character of the boundary integral operators, continuous  Galerkin approximations  are usually used on the boundary. As a consequence, the major difficulty that is encountered in the design of a DG-FEM/BEM method (cf. \cite{GHS10, HMS16, Cockburn} and \cite[Section 4]{Daveau}) is the mismatch that occurs between the interior and the boundary unknowns on the transmission interface. In our case, this difficulty manifests itself in the transmission condition \eqref{ModelProblemFemBEM3} where we have two variables of different nature. From the one side (as the discrete variable representing $\psi$ is $\rmH^1(\Gamma)$-conforming) we have a globally surface-divergence free function, from the other hand, the tangential trace of the DG approximation of the magnetic field is not $\rmH(\text{div}_\Gamma)$-conforming. This impedes one to merge the two variables at the discrete level as in \cite{MS03}. To address this problem, we exploit the ability of DG-methods to incorporate essential boundary conditions into the variational formulation and impose \eqref{ModelProblemFemBEM3} weakly. As a result, in comparison with \cite{MS03}, we have one further independent unknown on the boundary. We show that the resulting IPDG-FEM/BEM is uniformly stable with respect to the mesh parameter in an adequate DG-norm. Moreover, under suitable regularity assumptions, we provide quasi-optimal asymptotic error estimates.

We end this section with some of the notations that we will  use below.
Given a real number $r\geq 0$ and a polyhedron $\mathcal O\subset \mathbb R^d$, $(d=2,3)$,
we denote the norms and seminorms of the usual Sobolev space
$\rmH^r(\mathcal O)$ by $\|\cdot \|_{r,\mathcal O}$ and $|\cdot|_{r,\mathcal O}$ respectively. 
We use the convention $\rmL^2(\mathcal O):= \rmH^0(\mathcal O)$.
We recall that, for any $t \in [-1,\: 1 ]$, the spaces $\rmH^{t}(\partial \mathcal O)$
have an intrinsic definition (by localization) on the Lipschitz surface $\partial \mathcal O$
due to their invariance under Lipschitz coordinate transformations. Moreover, for all $0< t\leq 1$,
$\rmH^{-t}(\partial\mathcal O)$ is the dual of $\rmH^{t}(\partial\mathcal O)$ with respect
to the pivot space $\rmL^2(\partial \mathcal{O})$.  Finally we consider $\mathbf{H}(\curl, \mathcal O):=\{ \bv \in \rmL^2(\mathcal O)^3 \, : \, \curl \bv \in \rmL^2(\mathcal O)^3\}$ and endow it with its usual Hilbertian norm $\norm{\bv}_{\mathbf{H}(\curl, \mathcal O)}^2:=
\norm{\bv}_{0, \mathcal O}^2 + \norm{\curl \bv}_{0, \mathcal O}^2$.

\section{The model problem}\label{s2}
Let $\Omega\subset \mathbb{R}^3$ be a bounded polyhedral domain with a Lipschitz boundary $\Gamma$.
We denote by $\mathbf{n}$ the unit normal vector on $\Gamma$ that points towards
$\Omega^e:= \mathbb{R}^3\setminus \overline \Omega$.
For the sake of simplicity, we assume that $\Omega$ is simply connected and that $\Gamma$ is connected.
We consider the eddy current problem
\begin{equation}\label{ModelProblem}
\begin{array}{rcll}
\imath \omega \mu \bH + \curl\, \bE  &=& \boldsymbol 0         &\text{in $\Omega$}\\%[2ex]
 \bE  &=& \sigma^{-1} ( \curl\, \bH - \bj_e )         &\text{in $\Omega$}\\%[2ex]
        \bH \times \n &=& \nabla p \times \n &\text{on $\Gamma$}\\%[2ex]
        \mu \bH \cdot \n &=& \mu_0 \nder{p} &\text{on $\Gamma$}\\%[2ex]
        - \Delta p &=& 0 &\text{in $\Omega^e$}\\%[2ex]
        p &=& O\big(1/\abs{\x}\big)  &\text{as $\abs{\x}\to \infty$},
\end{array}
\end{equation}
where $\omega>0$ is the angular frequency, $\mu_0$ is the magnetic permeability of the free space, and the conductivity $\sigma$ and the magnetic permeability $\mu$ in the conductor $\Omega$ are positive and piecewise constant functions with respect to a partition of the domain $\Omega$ into Lipschitz polyhedra. 
Here $\bj_e$ denotes the (complex valued) applied current density, $\bE$ and $\bH$ are the electric field and the magnetic field respectively and $p$ is the scalar magnetic potential in the exterior region $\Omega_e$, namely, $\bH=\nabla p$ in $\Omega_e$.

A finite element formulation of problem \eqref{ModelProblem} requires the approximation of the asymptotic condition on $p$ at infinity by an homogeneous Dirichlet boundary condition on an artificial boundary $\Sigma$ located sufficiently far from the conductor $\Omega$. A more accurate strategy for solving problem \eqref{ModelProblem} consists  in reducing the computational domain to the conductor $\Omega$. This can be achieved by considering non-local boundary conditions provided by the following integral equations relating
the Cauchy data $\lambda := \nder{p}\quad \text{and} \quad \psi := p|_\Gamma$ on $\Gamma$ (see, e.g., \cite[Chap. 3]{sauterSchwab}):
\begin{align}
   \label{inteq1}
   \psi &= \left( \tfrac{1}{2}I + K\right) \psi - V \lambda \\
   \label{inteq2}
   \lambda &= -W\psi + \left(\tfrac{1}{2}I - K^{\texttt{t}}\right) \lambda
\end{align}
where $V$, $K$, $K^{\texttt{t}}$ are the boundary integral operators representing the single, double and adjoint of the
double layer, respectively, and $W$ is the hypersingular operator. This yields to an exact formulation of problem \eqref{ModelProblem}
that is adequate for a coupled FEM-BEM discretization strategy as: 
\begin{align}
 \label{ModelProblemFemBEM1}\imath \omega \mu \bH + \curl\, \left( \sigma^{-1} ( \curl\, \bH - \bj_e ) \right)  &= \boldsymbol 0         \qquad\text{in $\Omega$}\\[1ex]
         \label{ModelProblemFemBEM3}\bH \times \n &= \curl_\Gamma \psi \qquad\text{on $\Gamma$}\\[1ex]
        \label{ModelProblemFemBEM4} \frac{\mu}{\mu_0} \bH \cdot \n &= -W\psi + \left(\tfrac{1}{2}I - K^{\texttt{t}}\right) \lambda \qquad\text{on $\Gamma$}\\[1ex]
      \label{ModelProblemFemBEM5}V \lambda +    \left( \tfrac{1}{2}I - K\right) \psi &=  0  \qquad\text{on $\Gamma$},
\end{align}
where $\curl_\Gamma$ is the curl operator on the surface $\Gamma$. 

In \cite{MS03} it is shown that, using \eqref{ModelProblemFemBEM3}, the unknown $\psi$ can be eliminated from \eqref{ModelProblemFemBEM4} and \eqref{ModelProblemFemBEM5}, and that 
the weak formulation of the reduced problem admits a unique solution $(\bH, \lambda)\in \mathbf{H}(\curl, \Omega)\times \rmH_0^{-1/2}(\Gamma)$, where $\rmH^{-1/2}_0(\Gamma):= \set{\eta\in \rmH^{-1/2}(\Gamma); \, \, \dual{\eta, 1}_\Gamma=0}$. Here, $\dual{\cdot, \cdot}_\Gamma$ stands for the duality pairing between $\rmH^{-1/2}(\Gamma)$ and $\rmH^{1/2}(\Gamma)$.
Then $\psi\in \rmH^{1/2}(\Gamma)$ is uniquely determined, up to an additive constant, from  \eqref{ModelProblemFemBEM3}, so it is unique in $\rmH^{1/2}_0(\Gamma):= \set{\varphi \in \rmH^{1/2}(\Gamma); \int_\Gamma \varphi =0}$. 

Once the Cauchy data $\lambda$ and $\psi$ are known, the solution is computed in the exterior domain $\Omega^e$
by using the integral representation formula
\begin{equation*}\label{IntegralRep}
p(\x) = \int_{\Gamma} \disp\frac{\partial E(\abs{\x-\y})}{\partial \n_{\y}}\, \psi(\y)\, ds_{\y} -
\int_{\Gamma} E(\abs{\x-\y}) \lambda(\y)\, ds_{\y}  \quad\text{in $\Omega^e$},
\end{equation*}
where $E(\abs{\x}):= \frac{1}{4\pi}\frac{ 1}{\abs{\x}}$
is the fundamental solution of the Laplace operator. Let us recall some important properties of the boundary
integral operators, see \cite{sauterSchwab} for details.
They are formally defined at almost every point $\x \in \Gamma$ by
\[
 V\xi (\x):= \int_\Gamma E(\abs{\x-\y}) \xi(\y)\, ds_{\y}, \qquad K\varphi (\x):= \int_\Gamma
 \disp\frac{\partial E(\abs{\x-\y})}{\partial \n_{\y}}\, \varphi(\y)\, ds_{\y},
\]
\[
 K^{\texttt{t}}\xi (\x):= \int_\Gamma
 \disp\frac{\partial E(\abs{\x-\y})}{\partial \n_{\x}}\, \xi(\y)\, ds_{\y}, \quad
 W \varphi (\x) := - \disp\frac{\partial}{\partial \n_{\x}} \int_\Gamma
 \disp\frac{\partial E(\abs{\x-\y})}{\partial \n_{\y}}\, \varphi(\y)\, ds_{\y}.
\]
They are bounded as mappings $V:\, \rmH^{-1/2}(\Gamma)\to \rmH^{1/2}(\Gamma)$,
$K:\, \rmH^{1/2}(\Gamma) \to \rmH^{1/2}(\Gamma)$ and $W:\, \rmH^{1/2}\to \rmH^{-1/2}(\Gamma)$. %, cf. \cite{sauterSchwab}. 
Moreover, there exist constants $C_V>0$ and $C_W>0$ such that
\begin{equation}\label{eq-coercive-1}
 \dual{\bar \chi, V\chi}_\Gamma \,\ge\, C_V\, \norm{\chi}_{-1/2,\Gamma}^2\quad
 \forall\, \chi \,\in\,   \rmH^{-1/2}(\Gamma)
\end{equation}
and
 \begin{equation}\label{eq-coercive-2}
\dual{W \varphi, \bar\varphi}_\Gamma \,+\, \left|\int_\Gamma \varphi \, \right|^2
\,\ge\, C_W\,\norm{\varphi}_{1/2,\Gamma}^2\quad
 \forall\,\varphi \,\in\, \rmH^{1/2}(\Gamma).
 \end{equation}

%\section{Notations}
\section{The DG-FEM/BEM formulation}\label{section5}

We consider a sequence  $\{\cT_h\}_h$ of conforming  and shape-regular triangulations %(in the sense of Ciarlet) 
of  $\overline \Omega$. 
We assume that each partition $\cT_h$ consists of tetrahedra $K$ of diameter $h_K$ and unit outward normal to $\partial K$ denoted $\n_K$. 
We also assume that the meshes $\{\cT_h\}_h$ are aligned with the discontinuities of the piecewise constant coefficients $\sigma$ 
and $\mu$. The parameter $h:= \max_{K\in \cT_h} \{h_K\}$ represents the mesh size.

We denote by $\cF_h$ the set of faces of the tetrahedra of the mesh, by $\cF_h^0$ the sets of interior faces %of  $\cT_h$ 
and by $\cF_h^\Gamma$ the set of boundary faces.
%\[
%\cF_h^\Gamma:= \set{F = \overline K\cap\Gamma;\quad  K\in \cT_h{\color{blue} \hbox{ and } F \in \cF_h} }\,.
%\]
Clearly $\cF_h:=\cF_h^0 \cup\cF_h^\Gamma$. 
We notice that $\set{\cF_h^\Gamma}_h$ is a shape-regular family of triangulations of $\Gamma$ into triangles $T$ of diameter $h_T$; therefore from now on we will denote by $T$ the faces on $\Gamma$.

Let $\mathcal{O}_h$ be either $\mathcal T_h$ or $\mathcal F_h^\Gamma$ and $E$ be a generic element of $\mathcal O_h$. We introduce for any  $s\geq 0$ the broken Sobolev spaces
\[
 \rmH^s(\mathcal{O}_h) := \prod_{E\in \mathcal{O}_h} \rmH^s(E)\quad \text{and} \quad  \mathbf{H}^s(\mathcal{O}_h) := \prod_{E\in \mathcal{O}_h} \rmH^s(E)^3\, .
 \]
For each $w:= \set{w_E}\in \rmH^s(\mathcal{O}_h)$, the components  $w_E$ 
represents the restriction $w|_E$. When no confusion arises, the restrictions will be written
without any subscript.
The space 
$\rmH^s(\mathcal{O}_h)$ is endowed with the Hilbertian norm
\[
\norm{w}_{s,\mathcal{O}_h}^2 := \sum_{E\in \mathcal{O}_h} \norm{w_E}^2_{s,E}.
\]
We consider identical definitions for the norm and the seminorm of the vectorial version $\mathbf{H}^s(\mathcal{O}_h)$.
We use the standard conventions $\rmL^2(\mathcal{O}_h):=\rmH^0(\mathcal{O}_h)$ and 
$\mathbf{L}^2(\mathcal{O}_h):=\mathbf{H}^0(\mathcal{O}_h)$ and introduce the  bilinear forms
\[
(w, z)_{\mathcal{O}_h} = \sum_{E\in \mathcal{O}_h} \int_E w_E z_E, 
\quad
\text{and}
\quad
(\mathbf{w}, \mathbf{z})_{\mathcal{O}_h} = \sum_{E\in \mathcal{O}_h} \int_E \mathbf{w}_E\cdot \mathbf{z}_E.
\]

Hereafter, given an integer $k\geq 0$ and a domain $D\subset \mathbb{R}^3$, $\cP_k(D)$ denotes the space of polynomials of degree at most $k$ on $D$. Let $h_\cF\in \prod_{F\in \mathcal{F}_h} \cP_0(F)$ be defined by 
$h_\cF|_F := h_F$ $, \forall F \in \mathcal{F}_h$, where $h_F$ represents the diameter of the face $F$. 
We also introduce $\mathtt{s}_\cF\in \prod_{F\in \mathcal{F}_h} \cP_0(F)$ defined by 
$\mathtt{s}_F := \min(\sigma|_K, \sigma|_{K'})$, if $F = \partial K \cap \partial K'\in \cF_h^0$ and  
$\mathtt{s}_F := \sigma|_K$,  if $F = \partial K \cap \Gamma\in \cF_h^\Gamma$.

%\section{The DG-FEM/BEM formulation}\label{section5}

We introduce, for $m\geq 1$, the finite element spaces
\[
 \mathbf{X}_h := \prod_{K\in \cT_h}\cP_m(K)^3\, , \quad  \Lambda_h := \set{\lambda\in  \prod_{T\in \cT_h^\Gamma}\cP_{m-1}(T); \quad \int_\Gamma \lambda = 0 }
 \]
and
\[
 \Psi_h := \set{\phi \in \mathcal C^0(\Gamma); \,\, \phi|_T \in \cP_{m+1}(T)\,\, \forall T\in \cF_h^\Gamma, \,\, \int_{\Gamma} \phi = 0}.
\]

Given $\bv\in \mathbf{H}^{1+s}(\cT_h)$, with $s>1/2$, we consider 
  $\curl_h\bv\in \mathbf{H}^s(\cT_h)$ given by 
$
 (\curl_h \bv)|_K = \curl\, \bv_K$, for all $K\in \cT_h$ and introduce
\[
\mathbf{H}^s(\curl, \cT_h) := \set{\bv \in \mathbf{H}^s(\cT_h);\quad \curl_h \bv \in \mathbf{H}^{s}(\cT_h)}.
\]

For $(\bv,\varphi) \in \mathbf{H}^s(\cT_h)\times \rmH^{1}(\cF_h^\Gamma)$, $s>1/2$, 
we introduce the jumps
$\jump{(\bv,\varphi)}$ by
\[
               \jump{(\bv,\varphi)}: = \begin{cases}
                \jump{\bv\times \n}_F :=\bv_K \times \n_K + \bv_{K'}\times \n_{K'} & \text{if $F=K\cap K'\in \cF_h^0(\Omega)$}\\
                \bv|_T\times \n - \curl_T\varphi & \text{if $T\in \cF_h^{\Gamma}$}
               \end{cases}
\]
and the averages $\mean{\bv}\in \mathbf{L}^2(\cF_h)$ by
\[
\mean{\bv} =
\begin{cases}
(\bv_K + \bv_{K'})/2 & \text{if $F=K\cap K'\in \cF_h^0$}\\
\bv_K & \text{if $T\subset \partial K \in \cF_h^\Gamma$}
\end{cases}.
\]

In order to derive the DG-FEM/BEM discretization  of \eqref{ModelProblemFemBEM1}-\eqref{ModelProblemFemBEM5} we 
assume that $\bH\in \mathbf{H}(\curl,\Omega)\cap  \mathbf{H}^s(\curl, \cT_h)$ and $\bj_e\in \mathbf{H}^s(\cT_h)$ with 
$s>1/2$. We test
\eqref{ModelProblemFemBEM4} with  $\varphi \in \rmH^{1/2}(\Gamma)\cap \rmH^{1}(\cF_h^\Gamma)$ 
\[
\left\langle -W\psi + \left(\tfrac{1}{2}I - K^{\texttt{t}}\right) \lambda, \varphi \right \rangle_\Gamma = \left\langle\frac{\mu}{\mu_0} \bH \cdot \n, \varphi \right\rangle_\Gamma,
\]
and use \eqref{ModelProblemFemBEM1} together with an integration by parts on $\Gamma$ to obtain 
\begin{equation}\label{inteqq1}
\left\langle -W\psi + \left(\tfrac{1}{2}I - K^{\texttt{t}}\right) \lambda, \varphi \right \rangle_\Gamma
=
 \frac{-1}{\i\omega\mu_0}\int_\Gamma \curl\, \bE \cdot \n \, \varphi 
 =\frac{-1}{\i\omega\mu_0}\int_\Gamma \bE\cdot \curl_\Gamma \varphi,
\end{equation}
where, for economy of notations,  we reintroduced here the electric field $\bE:=\sigma^{-1} (\curl\, \bH - \bj_e)$. 
Moreover, we deduce from \eqref{ModelProblemFemBEM1} that, 
for all $\bv \in \mathbf{H}^s(\curl, \mathcal T_h)$,  
\begin{equation}\label{aaa1}
\sum_{K\in \cT_h} \Big( \int_K( \i \omega \mu \bH \cdot \bv + \bE \cdot \curl\, \bv )+ \int_{\partial K} \bE\cdot \bv \times \n_K \Big)=0 ,
\end{equation}
We also obtain from \eqref{ModelProblemFemBEM1}  that $\curl\, \bE \in \rmL^2(\Omega)^3$. Consequently, the jumps of the tangential components of $\bE \in \mathbf{H}^s(\cT_h) \cap \mathbf{H}(\curl, \Omega)$ vanish across the internal faces $F\in \cF^0_h$ and
\begin{equation*}
\sum_{K\in \cT_h} \int_{\partial K} \bE\cdot \bv \times \n_K  = 
\sum_{F\in \cF^0_h} \int_{F} \mean{\bE}\cdot \jump{\bv\times \n}
+\sum_{T \in \mathcal F_h^\Gamma} 	\int_{T} \bE\cdot \bv \times \n.
\end{equation*}
Inserting this identity in \eqref{aaa1} and adding the resulting equation to \eqref{inteqq1}, since for $T\in \cF_h^{\Gamma}$ one has $\jump{(\bv,\varphi)}_{|T}=\bv|_T\times \n - \curl_T\varphi $ we easily get
\begin{multline}\label{bbb1}
(\i \omega \mu \bH, \bv)_{\cT_h} + (\bE, \curl_h \bv)_{\cT_h} +
\dual{\mean{\bE}, \jump{(\bv,\varphi)}}_{\cF_h}\\
+ \i \omega\mu_0 \left\langle W\psi - \left(\tfrac{1}{2}I - K^{\texttt{t}}\right) \lambda, \varphi \right \rangle_\Gamma= 0 .
\end{multline}
Finally, testing \eqref{ModelProblemFemBEM5} with $\eta \in \rmH^{-1/2}(\Gamma)$ gives,
\begin{equation}\label{ccc1}
\imath \omega \mu_0 \dual{\eta, \left(\tfrac{1}{2}I - K\right) \psi}_\Gamma+ \imath \omega \mu_0\dual{\eta, V\lambda}_\Gamma = 0.
\end{equation}

Inspired from \eqref{bbb1} and \eqref{ccc1} we propose the following DG-FEM/BEM formulation for problem \eqref{ModelProblem}:
Find $(\bu_h, \psi_h)\in \mathbf{X}_h \times \Psi_h$ and $\lambda_h\in  \Lambda_h$ such that
\begin{equation}\label{ldg-FemBem}
 \begin{array}{rclcl}
 A_h((\bu_h, \psi_h),(\bv,\varphi) ) &-&\imath \omega \mu_0\, \dual{\lambda_h, (\tfrac{1}{2}I - K) \varphi}_\Gamma &=& L_h((\bv,\varphi) ) \\[2ex]
 \imath \omega \mu_0\, \dual{\eta, (\tfrac{1}{2}I - K) \psi_h}_\Gamma &+& \imath \omega\, \mu_0\dual{\eta, V\lambda_h}_\Gamma &=& 0,
 \end{array}
\end{equation}
for all $(\bv, \varphi)\in \mathbf{X}_h \times \Psi_h$ and $\eta\in \Lambda_h$, 
where 
\begin{align*}
 A_h((\bu_h, \psi_h),(\bv,\varphi) ) &:= 
 \imath \omega (\mu \bu_h,\ \bv)_{\cT_h} + (\sigma^{-1} \curl_h \bu_h, \curl_h \bv)_{\cT_h} + \imath \omega \mu_0 \dual{W \psi_h, \varphi}_{\Gamma}\\[1ex]
  &+ \dual{\mean{\sigma^{-1}\curl_h \bu_h}, \jump{(\bv,\varphi) }}_{\cF_h} 
   - \overline{\dual{\mean{\sigma^{-1}\curl_h \bv}, \jump{(\bu_h, \psi_h)}}}_{\cF_h}\\
  &+  \alpha \dual{\mathtt{s}^{-1}_{\cF} h_{\cF}^{-1} \jump{(\bu_h, \psi_h)},\jump{(\bv,\varphi) }}_{\cF_h},
\end{align*}
with a parameter $\alpha\geq 0$ and
\[
 L_h((\bv,\varphi) ) := (\sigma^{-1}\bj_e, \curl_h \bv)_{\cT_h} + \left\langle \mean{\sigma^{-1}\bj_e}, \jump{(\bv,\varphi) } \right\rangle_{\cF_h}.
\]

The following proposition shows that the DG-FEM/BEM scheme \eqref{ldg-FemBem} is consistent.

\begin{prop} %\label{consistency0B}
Let $((\bH, \psi), \lambda)\in [\mathbf{H}(\curl,\Omega)\times \rmH_0^{1/2}(\Gamma)]\times \rmH_0^{-1/2}(\Gamma)$ be the solution of \eqref{ModelProblemFemBEM1}-\eqref{ModelProblemFemBEM5}. Assume that $\sigma^{-1} \bj_e\in \mathbf{H}^{s}(\cT_h)$  and that $(\bH,\psi)\in \mathbf{H}^s(\curl, \cT_h)\times \rmH^{1}(\cF_h^\Gamma)$,  with $s>1/2$, then
	\begin{equation*}\label{diffBEM}
	 \begin{array}{rclcl}
	 A_h((\bH,\psi),(\bv,\varphi) ) &-&\imath \omega \mu_0 \dual{\lambda, (\tfrac{1}{2}I - K) \varphi}_\Gamma &=& L_h((\bv,\varphi) ), \\[2ex]
	 \imath \omega \mu_0 \dual{\eta, (\tfrac{1}{2}I - K) \psi}_\Gamma &+& \imath \omega \mu_0\dual{\eta, V\lambda}_\Gamma &=& 0,
	 \end{array}
	\end{equation*}
	for all $(\bv, \varphi)\in \mathbf{X}_h \times \Psi_h$ and $\eta\in \Lambda_h$.
\end{prop}
\begin{proof}
The result is a direct consequence of identities \eqref{bbb1} and \eqref{ccc1}, having used the fact that $\bH \in \mathbf{H}(\curl,\Omega)$, thus that $\jump{\bH \times \n}_F=0$ for each $F \in {\mathcal F}_h^0$, and equation \eqref{ModelProblemFemBEM3}, giving $\bH_{|T} \times \n = \curl_T \psi$ for each $T  \in {\mathcal F}_h^\Gamma$.
\end{proof}

\section{Convergence analysis}\label{section6}
We introduce the bilinear form
\begin{multline*}
\mathbb{A}_h\left( ((\bu, \phi), \zeta), ((\bv,\varphi) , \eta) \right):= A_h((\bu, \phi),(\bv,\varphi) ) -\imath \omega \mu_0 \dual{\zeta, (\tfrac{1}{2}I - K) \varphi}_\Gamma\\ 
+ \imath \omega \mu_0 \dual{\bar\eta, (\tfrac{1}{2}I - K) \bar\phi}_\Gamma+ \imath \omega \mu_0 \dual{\bar\eta, V\bar\zeta}_\Gamma
\end{multline*}
and 
 define in $\big(\mathbf{H}^s(\curl, \cT_h)\times [\rmH^{1/2}(\Gamma)\cap\rmH^1(\cF_h^\Gamma)]\big)\times \rmH^{-1/2}(\Gamma)$ 
the norms 
\begin{multline*}
\trinorm{((\bv,\varphi) , \eta)}^2 := \norm{(\omega\mu)^{1/2}\bv}^2_{0,\Omega} + \norm{\sigma^{-1/2}\curl_h \bv}^2_{0,\Omega} 
+ \norm{\mathtt{s}^{-1/2}_{\cF} h_{\cF}^{-1/2}\jump{(\bv,\varphi) }}^2_{0,\cF_h}\\ + \omega\mu_0\norm{\varphi}^2_{1/2,\Gamma} + \omega\mu_0\norm{\eta}^2_{-1/2,\Gamma}
\end{multline*}
and
\[
\trinorm{((\bv,\varphi) , \eta)}_{\ast}^2 :=  \trinorm{((\bv,\varphi) , \eta)}^2 + \norm{\mathtt{s}_{\cF}^{1/2} h_{\cF}^{1/2} \mean{\sigma^{-1}\curl_h \bv}}^2_{0,\cF_h}.
\]

The following discrete trace inequality is standard, (see, e.g., \cite[Lemma 1.46]{DiPietroErn}).
\begin{lemma}
 For all integer $k\geq 0$ there exists a constant $C^*>0$, independent of $h$, such that,
 \begin{equation}\label{discreteTrace3D} 
  h_K \norm{ v}^2_{0,\partial K} \leq C^* \norm{ v}^2_{0,K} \quad \forall v\in \cP_k(K),\quad 
  \forall K\in \cT_h.
 \end{equation}
\end{lemma}
It allow us to prove the following result.
\begin{lemma}\label{equiv}
For all $k\geq 0$, there exist a constants $C_\Omega>0$, independent 
of the mesh size and the coefficients, such that  
\begin{equation}\label{discIneq1}
\norm{\mathtt{s}_{\cF}^{1/2} h_{\cF}^{1/2} \mean{\sigma^{-1}\mathbf{w}}}_{0,\cF_h} \leq C_\Omega \norm{\sigma^{-1/2} \mathbf{w}}_{0,\Omega},\quad \mathbf{w} \in \prod_{K\in \cT_h}\cP_k(K)^3
\end{equation}
\end{lemma}
\begin{proof}
By definition of $\mathtt{s}_{\cF}$, for any $\mathbf{w} \in \prod_{K\in \cT_h}\cP_k(K)^3$,
\begin{multline*}
\norm{\mathtt{s}_{\cF}^{1/2} h_{\cF}^{1/2} \mean{\sigma^{-1}\mathbf{w}}}^2_{0,\cF_h} =  
\sum_{F\in \cF_h} h_F \norm{ \mathtt{s}_F^{1/2}\mean{\sigma^{-1}\mathbf{w}}_F }^2_{0,F}\\ 
\leq \sum_{K\in \cT_h} \sum_{F\in \cF(K)} h_F \norm{ \mathtt{s}_F^{1/2}\sigma_K^{-1}\mathbf{w}_K }^2_{0,F}
\leq  \sum_{K\in \cT_h}  h_K 
\norm{ \sigma_K^{-1/2}\mathbf{w}_K }^2_{0,\partial K},
\end{multline*}
where $\cF(K)$ denotes the set of faces composing the boundary of $K$, namely, $\cF(K) := \set{F\in \cF_h;\quad F\subset \partial K}$. 
Then the result follows from \eqref{discreteTrace3D}.
\end{proof}

\begin{prop}\label{boundednessBEM}
There exists a constant $M^*>0$ independent of $h$ such that 
\[
| \mathbb{A}_h\left( ((\bu, \phi), \zeta), ((\bv,\varphi) , \eta) \right) | \leq M^* \trinorm{((\bu, \phi), \zeta)}_{\ast} \trinorm{((\bv,\varphi) , \eta)}
\]
for all $((\bv,\varphi),\eta)\in (\mathbf{X}_h \times \Psi_h) \times \Lambda_h$ and for all $((\bu, \phi),\zeta)\in \big(\mathbf{H}^s(\curl, \cT_h)\times [\rmH^{1/2}(\Gamma) \cap \rmH^{1}(\cF_h^\Gamma)] \big)\times H^{-1/2}(\Gamma)$, 
with $s>1/2$. 
\end{prop}
\begin{proof}
The result follows immediately from the Cauchy-Schwarz inequality, the boundedness of the mappings $V:\, \rmH^{-1/2}(\Gamma)\to \rmH^{1/2}(\Gamma)$, $K:\, \rmH^{1/2}(\Gamma) \to \rmH^{1/2}(\Gamma)$ and $W:\, \rmH^{1/2}(\Gamma)\to \rmH^{-1/2}(\Gamma)$
and from the fact that  the norms $\trinorm{\cdot}$ and $\trinorm{\cdot}_{\ast}$ are equivalent in $(\mathbf{X}_h \times \Psi_h) \times \Lambda_h$ (as a consequence of  Lemma~\ref{equiv}). 
\end{proof}

\begin{prop}\label{discElipBEM} 
There exists a constant $\beta^*>0$ ,independent of $h$, such that  
\[
\text{Re}\left[ (1 - \imath) \mathbb{A}_h\left( ((\bv,  \varphi), \eta), ((\bar \bv, \bar \varphi), \bar\eta)\right)\right] \geq \beta^*  \trinorm{((\bv, \varphi) \eta)}^2
\]
for all $((\bv,\varphi), \eta)\in (\mathbf{X}_h\times \Psi_h)\times \Lambda_h$.
\end{prop}
\begin{proof}
By definition of $\mathbb{A}_h(\cdot, \cdot)$, 
\begin{multline*}
\text{Re}\left[ (1 - \imath) \mathbb{A}_h\left( ((\bv,  \varphi), \eta), ((\bar \bv, \bar \varphi), \bar\eta)\right)\right] = \omega \norm{\mu^{1/2} \bv}_{0,\Omega}^2 + \norm{\sigma^{-1/2} \curl_h \bv}^2_{0,\Omega} 
\\+ \alpha \norm{\mathtt{s}^{-1/2}_{\cF} h_{\cF}^{-1/2}\jump{(\bv,\varphi) _h} }^2_{0,\cF_h}
 +  \omega \mu_0\dual{\eta, V\bar \eta}_\Gamma 
+  \omega \mu_0 \dual{W \bar\varphi, \varphi}_{\Gamma}
\end{multline*}
and using \eqref{eq-coercive-1} and \eqref{eq-coercive-2} we deduce the result with $\beta^* = \min(1, \alpha, C_V, C_W)$.
\end{proof}

We deduce readily from the consistency of  the DG-FEM/BEM scheme, Proposition \ref{boundednessBEM} and  Proposition \ref{discElipBEM} the following C\'ea estimate.

\begin{theorem}\label{ceaBem}
Assume that $\sigma^{-1} \bj_e\in \mathbf{H}^{s}(\cT_h)$, with $s>1/2$. Then, the DG-FEM/BEM formulation \eqref{ldg-FemBem} has a unique solution for any parameter $\alpha \geq 0$.
Moreover if $((\bH, \psi), \lambda)\in [\mathbf{H}(\curl,\Omega)\times H_0^{1/2}(\Gamma)]\times H_0^{-1/2}(\Gamma)$ and $((\bu_h, \psi_h), \lambda_h)\in ((\mathbf{X}_h\times \Psi_h)\times \Lambda_h)$ are the solutions of \eqref{ModelProblemFemBEM1}-\eqref{ModelProblemFemBEM5}
and \eqref{ldg-FemBem} respectively,  and $(\bH,\psi)\in \mathbf{H}^s(\curl, \cT_h)\times \rmH^{1}(\cF_h^\Gamma)$,  with $s>1/2$, then,  
\begin{equation*}
\trinorm{((\bH - \bu_h, \psi - \psi_h), \lambda - \lambda_h)} \leq (1 + \frac{M^*}{\beta^*}) 
\trinorm{((\bH - \bv, \psi - \varphi), \lambda - \eta)}_*,
\end{equation*}
for all $((\bv, \varphi),\eta)\in (\mathbf{X}_h\times \Psi_h)\times \Lambda_h$.
\end{theorem}

\section{Asymptotic error estimates}
We denote by $\bPi_{h}^{\text{curl}}$ the $m$-order $\mathbf{H}(\curl, \Omega)$-conforming N\'ed\'elec interpolation operator of the second kind, see for example \cite{NED86, Boffi} or \cite[Section 8.2]{Monk}.
It is well known that $\bPi_h^{\text{curl}}$ is bounded on 
$\mathbf{H}(\curl, \Omega)\cap \mathbf{H}^s(\curl, \cT_h)$ for $s>1/2$. Moreover, there exists a constant $C>0$ independent of $h$ such that (cf. \cite{Monk})
\begin{equation}\label{errorInterp1}
\norm{\bu - \bPi_{h}^{\text{curl}} \bu}_{\mathbf{H}(\curl, \Omega)}  \leq C h^{\min(s, m)} \big(  \norm{\bu}_{s, \cT_h} + \norm{\curl_h \bu}_{s, \cT_h} \big).
\end{equation}

For all triangle $T\in \cF_h^{\Gamma}$ we define the interpolation operator $\pi^\Gamma_{T}:\rmH^{1/2+s}(T)\to \cP_{m+1}(T)$, $s>1/2$, uniquely determined by the conditions
\begin{equation}\label{freedomV}
\pi^\Gamma_{T} \varphi (\mathbf{a}_T) = \varphi(\mathbf{a}_T),\quad \text{for all vertex $\mathbf{a}_T$ of $T$},
\end{equation}
\begin{equation}\label{freedomEb2d}
\int_e \pi^\Gamma_{T} \varphi q = \int_e \varphi q \qquad \forall q \in \cP_{m-1}(e), \quad 
\text{for all edge $e$ of $T$},
\end{equation}
\begin{equation}\label{freedomTb2d}
\int_T \pi^\Gamma_{T} \varphi q = \int_T \varphi q \qquad \forall q \in \cP_{m-2}(T).
\end{equation}
The corresponding global interpolation operator $\pi^\Gamma_{h}$ is $\rmH^1(\Gamma)$-conforming 
and satisfies the following interpolation error estimate. 
\begin{lemma}\label{varphiHmedio}
Assume that $\varphi\in H^{1/2+s}(\cF_h^\Gamma)\cap H^{1/2}(\Gamma)$ with $s>1/2$, then
 \begin{equation}\label{interpG}
  \norm{\varphi - \pi^\Gamma_{h} \varphi}_{t,\Gamma} \leq C h^{\min\{1/2+s,m+2\}-t} 
  \norm{\varphi}_{1/2+s,\cF_h^\Gamma},\quad t\in \set{0,1,1/2}
 \end{equation}
 with a constant $C>0$ independent of $h$.
\end{lemma}
\begin{proof}
We notice that, as $s>1/2$, $\rmH^{1/2+s}(\cF_h^\Gamma)\cap \rmH^{1/2}(\Gamma)\subset \mathcal{C}^0(\Gamma)$. Hence, 
$\pi_h^\Gamma$ is bounded on $\rmH^{s+1/2}(\cF_h^\Gamma)\cap \rmH^{1/2}(\Gamma)$.
The interpolation error estimates for $t=0$ and $t=1$ are standard. The case $t=1/2$ is obtained from the interpolation 
inequality
\[
\norm{\phi}_{1/2, \Gamma}^2 \leq \norm{\phi}_{0, \Gamma} \norm{\phi}_{1, \Gamma} \quad \forall \phi \in \rmH^1(\Gamma).
\] 
\end{proof} 

We introduce 
$
\mathbf{L}^2_t(\Gamma) = \set{\boldsymbol{\varphi}\in \rmL^2(\Gamma)^3;\,\,  \boldsymbol{\varphi}\cdot \n = 0}
$
and consider the $m$-order order Brezzi-Douglas-Marini (BDM) (see \cite{Boffi, Monk}) finite element approximation of  
\[
\mathbf{H}(\text{div}_\Gamma, \Gamma) := \set{ \boldsymbol{\varphi}\in \mathbf{L}_t^2(\Gamma); \quad 
\text{div}_\Gamma \boldsymbol{\varphi}\in L^2(\Gamma) }
\]
relatively to the mesh $\cF_h^\Gamma$, where $\div_\Gamma$ is the divergence operator on the surface $\Gamma$. 
It is given by 
\[
\mathrm{BDM}(\cF_h^\Gamma) = \set{\boldsymbol{\varphi}\in  \mathbf{H}(\text{div}_\Gamma, \Gamma); \quad 
\boldsymbol{\varphi}|_T \in \cP_m(T)^2,\quad \forall T\in \cF_h^\Gamma}.
\]

The corresponding interpolation operator $\Pi_h^{\text{BDM}}$ is bounded on $
\mathbf{H}(\text{div}_\Gamma, \Gamma)\cap \prod_{T\in \cF_h^\Gamma} \rmH^\delta(T)^2$ for all $\delta>0$, and it is not difficult to check that  is 
related to $\Pi_h^{\text{curl}}$ through the following commuting diagram property:% {\color{red} (see, e. g., \cite{NED86})}:
\begin{equation}\label{comm1}
(\bPi_{h}^{\text{curl}} \bv) \times \n = \Pi_{h}^{\text{BDM}} ( \bv \times \n )
\quad \forall \bv \in \mathbf{H}(\curl, \Omega)\cap \mathbf{H}^s(\curl, \cT_h), \quad s>1/2.
\end{equation} 

Moreover the following result holds true.

\begin{prop}
Let $((\bH, \psi), \lambda)\in [\mathbf{H}(\curl,\Omega)\times \rmH_0^{1/2}(\Gamma)]\times \rmH_0^{-1/2}(\Gamma)$ is the solution of \eqref{ModelProblemFemBEM1}-\eqref{ModelProblemFemBEM5}. Assume that 
$(\bH, \psi)\in \mathbf{H}^s(\curl, \cT_h)\times H^{1/2+s}(\cF_h^\Gamma)$ with $s>1/2$. Then,
\begin{equation}\label{comm0}
(\bPi_{h}^{\mathrm{curl}} \bH) \times \n = \curl_\Gamma (\pi^\Gamma_{h} \psi) \quad  \text{on $\Gamma$}.
\end{equation}
\end{prop}
\begin{proof}
Let us first prove that 
\begin{equation}\label{comm2}
\Pi_{h}^{\text{BDM}} (\curl_\Gamma \psi )
=\curl_\Gamma ( \pi^\Gamma_{h} \psi).
\end{equation}
It is clear that $\curl_\Gamma  \pi^\Gamma_{h} \psi\in \mathrm{BDM}(\cF_h^\Gamma)$ and it can be shown that 
the tangential fields $\Pi_{h}^{\text{BDM}} (\curl_\Gamma \psi )$ and $\curl_\Gamma  \pi^\Gamma_{h} \psi$ have the same 
\textrm{BDM}-degrees of freedom in each $F\in \cF_h^\Gamma$, 
which gives \eqref{comm2}. We deduce now \eqref{comm0} from \eqref{comm1}, \eqref{comm2} and the transmission condition 
\eqref{ModelProblemFemBEM3}.
\end{proof}

We will also use the best $L^2(\Gamma)$ approximation in $\Lambda_h$ of a function $\eta \in H^r(\cF_h^\Gamma)$, with $r>0$,  the $\mathbf{L}^2(\cT_h)$-orthogonal projection onto $\prod_{K\in \cT_h} \cP_{m-1}(K)^3$ of a function ${\bf w}\in {\bf H}^r({\mathcal T}_h)$, with $r>1/2$, and the following estimates.

\begin{lemma}
Assume that $\eta\in H^r(\cF_h^\Gamma)$ for some $r \geq 0$. 
Then,
\begin{equation}\label{sigman}
 \norm{\eta - \pi_{\Lambda_h}\eta}_{-1/2,\Gamma}\leq C h^{\min\{ r,m \}+1/2} \norm{\eta}_{r,\cF_h^\Gamma},
\end{equation}
where $\pi_{\Lambda_h}\eta$ the best $L^2(\Gamma)$ approximation of $\eta$ in $\Lambda_h$. 
\end{lemma}
\begin{proof}
 See \cite[Theorem 4.3.20]{sauterSchwab}.
\end{proof}

\begin{lemma}\label{v}
Let ${\bf P}_h^{m-1}$ be the $\mathbf{L}^2(\cT_h)$-orthogonal projection onto $\prod_{K\in \cT_h} \cP_{m-1}(K)^3$. For all $K\in \cT_h$ and  $\mathbf{w}\in \mathbf{H}^{r}(K)$, $r> 1/2$, we have
 \begin{equation}\label{proj}
   h_K^{1/2}\norm{\mathbf{w} - {\bf P}^{m-1}_h \mathbf{w}}_{0,\partial K}  + \norm{\mathbf{w} - {\bf P}^{m-1}_h \mathbf{w}}_{0,K} 
  \leq C h_K^{\min\{r,m\}} \norm{\mathbf{w}}_{r,K},
 \end{equation}
 with a constant $C>0$ independent of $h$. 
\end{lemma}
\begin{proof}
See \cite[Lemma 1.58 and Lemma 1.59]{DiPietroErn}.
\end{proof}

We are now ready to prove the main theorem of this section.
\begin{theorem}\label{conv}
Let $((\bH, \psi), \lambda)\in [\mathbf{H}(\curl,\Omega)\times H_0^{1/2}(\Gamma)]\times H_0^{-1/2}(\Gamma)$ be the solution of 
\eqref{ModelProblemFemBEM1}-\eqref{ModelProblemFemBEM5}
and $((\bH_h, \psi_h), \lambda_h)\in ((\mathbf{X}_h\times \Psi_h)\times \Lambda_h)$ the solutions of  \eqref{ldg-FemBem}. Assume that $\sigma^{-1} \bj_e\in \mathbf{H}^{s}(\cT_h)$  and that $(\bH,\psi)\in \mathbf{H}^s(\curl, \cT_h)\times \rmH^{s+1/2}(\cF_h^\Gamma)$ and $\lambda \in \rmH^{s-1/2}(\cF_h^\Gamma)$ with $s>1/2$. 
Then, there exists $C>0$ independent of $h$ such that 
\begin{multline*}
\trinorm{((\bH - \bH_h, \psi - \psi_h), \lambda - \lambda_h)} \leq C h^{\min(s, m)}
\Big(
\norm{\bH}_{s, \cT_h} + \norm{\curl\, \bH}_{s, \cT_h}
\\+
\norm{\psi}_{s+1/2,\cF_h^\Gamma}+
\norm{\lambda}_{s-1/2,\cF_h^\Gamma}
\Big).
\end{multline*}
\end{theorem}
\begin{proof}
We deduce from Theorem \ref{ceaBem} that 
\[
\trinorm{((\bH - \bH_h, \psi - \psi_h), \lambda - \lambda_h)} \leq (1 + \frac{M^*}{\beta^*}) 
\trinorm{((\bH - \bPi_h^{\text{curl}}\bH, \psi - \pi^\Gamma_h\psi), \lambda - \pi_{\Lambda_h}\lambda)}_*.
\]
By virtue of \eqref{comm0}, 
we have
\begin{multline}\label{t1}
\trinorm{((\bH - \bPi_h^{\text{curl}}\bH, \psi - \pi^\Gamma_h\psi), \lambda - \pi_{\Lambda_h}\lambda)}^2_* = 
\norm{(\omega\mu)^{1/2}(\bH - \bPi_h^{\text{curl}}\bH)}^2_{0,\Omega}\\ + \norm{\sigma^{-1/2}\curl_h (\bH - \bPi_h^{\text{curl}}\bH)}^2_{0,\Omega} + \omega\mu_0\norm{\psi - \pi^\Gamma_h}^2_{1/2,\Gamma} + \omega\mu_0\norm{\lambda - \pi_{\Lambda_h}\lambda}^2_{-1/2,\Gamma}\\ + 
\norm{\mathtt{s}_{\cF}^{1/2} h_{\cF}^{1/2} \mean{\sigma^{-1}\curl_h (\bH - \bPi_h^{\text{curl}}\bH)}}^2_{0,\cF_h}.
\end{multline}  

We deduce from the triangle inequality that,
\begin{multline}\label{t2}
\norm{\mathtt{s}_{\cF}^{1/2} h_{\cF}^{1/2} \mean{\sigma^{-1}\curl (\bH - \bPi^{\text{curl}}_h\bH)}}_{0,\cF_h} = 
\norm{\mathtt{s}_{\cF}^{1/2} h_{\cF}^{1/2} \mean{\sigma^{-1}(I - {\bf P}^{m-1}_h)\curl \,\bH)}}_{0,\cF_h} \\+ 
\norm{\mathtt{s}_{\cF}^{1/2} h_{\cF}^{1/2} \mean{\sigma^{-1}({\bf P}^{m-1}_h\curl \,\bH - \curl\, \bPi^{\text{curl}}_h\bH)}}_{0,\cF_h} = A_\Omega + B_\Omega\, .
\end{multline}
Using \eqref{discIneq1} yields
\begin{multline}\label{t3}
B_\Omega \leq C_\Omega \norm{\sigma^{-1/2} ({\bf P}^{m-1}_h\curl\, \bH - \curl\, \bPi^{\text{curl}}_h\bH) }_{0,\Omega}\\ = 
C_\Omega \norm{\sigma^{-1/2} {\bf P}^{m-1}_h \curl( \bH -  \bPi^{\text{curl}}_h\bH) }_{0,\Omega} \leq 
C_\Omega \norm{\sigma^{-1/2}\curl ( \bH - \bPi^{\text{curl}}_h \bH) }_{0,\Omega},
\end{multline}
and it is straightforward that  
\begin{equation}\label{t4}
A_\Omega^2 \leq  \sum_{K\in \cT_h}  h_K 
\norm{ \sigma_K^{-1/2}(\curl\,\bH - {\bf P}^{m-1}_h\curl\, \bH) }^2_{0,\partial K}.
\end{equation}
Combining \eqref{t1}, \eqref{t2}, \eqref{t3} and \eqref{t4}, we deduce that
\begin{multline*}
\trinorm{((\bH - \bPi_h^{\text{curl}}\bH, \psi - \pi^\Gamma_h\psi), \lambda - \pi_{\Lambda_h}\lambda)}^2_* \leq  
\norm{(\omega\mu)^{1/2}(\bH - \bPi_h^{\text{curl}}\bH)}^2_{0,\Omega}\\ 
+ (1+C_\Omega^2)\norm{\sigma^{-1/2}\curl (\bH - \bPi_h^{\text{curl}}\bH)}^2_{0,\Omega} + \omega\mu_0\norm{\psi - \pi^\Gamma_h\psi}^2_{1/2,\Gamma}\\ + \omega\mu_0\norm{\lambda - \pi_{\Lambda_h}\lambda}^2_{-1/2,\Gamma} + 
 \sum_{K\in \cT_h}  h_K 
\norm{ \sigma_K^{-1/2}(\curl\,\bH - {\bf P}^{m-1}_h\curl\, \bH) }^2_{0,\partial K}.
\end{multline*} 
Finally, applying the interpolation error estimates \eqref{errorInterp1}, \eqref{proj},  \eqref{interpG} 
and \eqref{sigman} we obtain
\begin{multline*}
\trinorm{((\bH - \bPi_h^{\text{curl}}\bH, \psi - \pi^\Gamma_h\psi), \lambda - \pi_{\Lambda_h}\lambda)}_* \leq 
C \Big(
h^{\min(s, m)}  (
\norm{\bH}_{s, \cT_h} + \norm{\curl\, \bH}_{s, \cT_h})\\ + h^{\min\{s+1/2,m+2\}-1/2} 
  \norm{\psi}_{s+1/2,\cF_h^\Gamma} + h^{\min\{ s-1/2,m \}+1/2} \norm{\lambda}_{s-1/2,\cF_h^\Gamma}
\Big), 
\end{multline*}
and the result follows. 
\end{proof}

\begin{remark}
It is well-known that different choices of finite elements could be chosen for the approximation of $\mathbf{H}(\curl,\Omega)$ and $H^{1/2}(\Gamma)$.
For instance, let us consider, for $m\geq 1$, 
$
 \mathbf{X}^{(0)}_h := \prod_{K\in \cT_h}\mathbf{ND}_m(K)$ and
 \[
 \Psi^{(0)}_h := \set{\phi \in \mathcal C^0(\Gamma); \,\, \phi|_T \in \cP_{m}(T)\,\, \forall T\in \cF_h^\Gamma, \, \int_{\Gamma} \phi = 0},
\]
where $\mathbf{ND}_m(K)\subset \cP_m(K)^3$ is the $m^{\text{th}}$-order (local) N\'ed\'elec finite element space of the first kind, see for example \cite{NED86, Boffi}. The DG-FEM/BEM scheme \eqref{ldg-FemBem} formulated in terms of the finite spaces $(\mathbf{X}^{(0)}_h \times \Psi^{(0)}_h)\times \Lambda_h$  provides, under the regularity assumption of Theorem \ref{conv}, the same order of convergence with less degrees of freedom. However, in this case, the non-standard basis functions of $\mathbf{ND}_m(K)$ are required for the implementation of the scheme. 
 
\end{remark}

%\subsubsection*{acknowledgement}
%Partial support by the University of Trento, and
%Spain's Ministry of Economy through Project MTM2013-43671-P.
%

\end{document}